\documentclass[12pt]{amsart}
\usepackage{amsmath, amsfonts, latexsym, array, amssymb}

\newtheorem{thm}{Theorem}[section]
\newtheorem{cor}[thm]{Corollary}
\newtheorem{lem}[thm]{Lemma}

\newtheorem{cla}[thm]{Claim}
\newtheorem{que}[thm]{Question}

\newtheorem{ex}[thm]{Example}

\begin{document}

\title[Trivial centralizer] {Open sets of diffeomorphisms with trivial centralizer in the $C^1$ topology}
\author{Lennard Bakker and Todd Fisher}
\address{Department of Mathematics, Brigham Young University, Provo, UT 84602}

\email{tfisher@math.byu.edu}
\email{bakker@math.byu.edu}

\date{May, 2014}
\commby{}

\begin{abstract} On the torus of dimension $2$, $3$, or $4$, we show that the subset of diffeomorphisms with trivial centralizer in the $C^1$ topology has nonempty interior. We do this by developing two approaches, the fixed point and the odd prime periodic point, to obtain trivial centralizer for an open neighbourhood of Anosov diffeomorphisms arbitrarily near certain irreducible hyperbolic toral automorphism.
 \end{abstract}

\maketitle

\section{Introduction}

The trivial centralizer problem, first posed in 1967 by Smale, asks about the topology of the subset of diffeomorphisms that commute only with their integer powers. For $M$ a compact, smooth manifold, and for $r\in{\mathbb N}\cup\{\infty\}$, the centralizer of $f\in\mathrm{Diff}^r(M)$ is the group
\[ Z(f)=\{g\in \mathrm{Diff}^r(M): fg=gf\}.\] 
We know that $Z(f)$ always contains the finite or countable subgroup $\langle f \rangle=\{ f^k:k\in{\mathbb Z} \}$.  We say that $f$ has  trivial centralizer if $Z(f)=\langle f\rangle$.  We denote the subset of $\mathrm{Diff}^r(M)$ with trivial centralizer by $\mathcal{T}^r(M)$. Smale's original question \cite{Sma1} about the trivial centralizer problem asked if ${\mathcal T}^r(M)$ is generic or residual in the $C^r$ topology, i.e., is it the countable intersection of open, dense subsets of ${\rm Diff}^r(M)$? By the turn of the century, the original question had expanded into a hierarchy of three questions about the possible size of ${\mathcal T}^r(M)$.  

\begin{que}\label{q.smale} $($Smale \cite{Sma2}$)$
\begin{itemize}
\item[(1)] Is $\mathcal{T}^r(M)$ dense in $\mathrm{Diff}^r(M)$? 
\item[(2)] Is $\mathcal{T}^r(M)$ residual in $\mathrm{Diff}^r(M)$? 
\item[(3)] Is $\mathcal{T}^r(M)$ open and dense in $\mathrm{Diff}^r(M)$?
\end{itemize}
\end{que}

Partial results to all of the three parts of Question~\ref{q.smale}  have been obtained over the years for various classes of diffeomorphisms. Kopell~\cite{Kop1} showed that $\mathcal{T}^r(S^1)$ is open and dense in $\mathrm{Diff}^r(S^1)$ for all $r\geq 2$. For higher dimensional manifolds and $r\geq 2$ there are answers under additional dynamical assumptions.  For instance, for certain Axiom A diffeomorphisms there have been several results obtained \cite{PY1, PY2, Fis2, Fis3}. For certain $C^{\infty}$ diffeomorphisms that are partially hyperbolic, Burslem obtained a residual set with trivial centralizer \cite{Bur}. In the $C^1$ setting Togawa \cite{Tog} proved that the generic Axiom A diffeomorphism has trivial centralizer.

In 2009, Bonatti, Crovisier, and Wilkinson~\cite{BCW1} proved that for any manifold $M$ the set $\mathcal{T}^1(M)$ is generic in $\mathrm{Diff}^1(M)$. Thus the second part (and hence the first part) of Question \ref{q.smale} has been answered in the affirmative for any $M$ when $r=1$. About the same time, Bonatti, Crovisier, Vago, and Wilkinson~\cite{BCVW} prove that for any $M$ there exists an open set $\mathcal{U}$ of $\mathrm{Diff}^1(M)$ and a dense set $\mathcal{D}$ of $\mathcal{U}$ such that each $f$ in $\mathcal{D}$ has uncountable, hence nontrivial, centralizer. Thus, the third part of Question \ref{q.smale} has been answered in the negative for any $M$ when $r=1$. These results point to another question about the topology of ${\mathcal T}^1(M)$, one that is already suggested in \cite{BCW1}.

\begin{que}\label{empty}
For what compact manifolds $M$ does $\mathcal{T}^1(M)$ have non-empty interior?
\end{que}

Among low dimensional manifolds, an answer to Question \ref{empty} for one particular manifold is already known. The aforementioned Bonatti, Crovisier, Vago, and Wilkinson  \cite{BCVW} showed that  $\mathcal{T}^1(S^1)$ has empty interior. The main objective of this paper is to show that for  low dimensional tori ${\mathbb T}^n$ we can answer yes to Question \ref{empty}.

\begin{thm}\label{t.open}
The set $\mathcal{T}^1(\mathbb{T}^n)$ has nonempty interior for $2\leq n\leq 4$.
\end{thm}

The first step in the proof of this Main Theorem, given in Section 2, is to identify the centralizer of a hyperbolic toral automorphism $A$. We show that each element of $Z(A)$ is an affine diffeomorphism of the form $B+c$ where $B$ is a toral automorphism that commutes with $A$, and $c$ is a fixed point of $A$. The group structure of $Z(A)$ depends on the centralizer $C(A)$ of a toral automorphism $A$ within the group of toral automorphisms. Using algebraic number theory and assuming that $A$ is irreducible, we identify the group structure of $C(A)$. We determine   conditions on $A$ under which we guarantee that
\[ C(A) = \langle A \rangle \times \langle J\rangle,\]
where $J$ is toral automorphism of finite even order that commutes with $A$. The point to emphasize here is that this group structure for $C(A)$ can only happen when $n=2$, $3$, or $4$, because by algebraic number theory, when $n\geq 5$, there is always another toral automorphism $B$ of infinite order that commutes with $A$ but is not an integer power of $A$. However, no matter the value of $n$, the even finite order cyclic subgroup $\langle J\rangle$ of $C(A)$ always contains the toral automorphism $-I$ of order $2$, where $I$ is the $n\times n$ identity matrix. We describe the group properties of $\langle J\rangle$, especially the relationship between $-I$ and $J$ within $\langle J\rangle$, as these play a key role.

The second step in the proof of the Main Theorem, given in Section 3, is to use the well known structural stability of an Anosov diffeomorphsim to get a $C^1$ open neighbourhood ${\mathcal U}(A)$ of Anosov diffeomorphisms containing a hyperbolic toral automorphism $A$. For each $f\in {\mathcal U}(A)$ there is a conjugating homeomorphism $h_f$ that satisfies $h_f f h_f^{-1} = A$. Under the inner automorphism induced by $h_f$,  the centralizer $Z(f)$ is isomorphic to a subgroup of $Z(A)$. So when $C(A) = \langle A\rangle \times \langle J\rangle$, then each $g\in Z(f)$ satisfies
\[ h_f g h_f^{-1} = A^mJ^l+c\]
for $m\in{\mathbb Z}$, $l$ a nonnegative integer no bigger than the order of $J$, and $c$ a fixed point of $A$. If we can show that $l=0$ and $c=0$, then we have trivial centralizer for $f$. To achieve this, we look within ${\mathcal U}(A)$ at the open neighbourhoods ${\mathcal U}_k(A)$ in which distinct $k$-periodic orbits have different spectra. These neighbourhoods serve as the candidates for the open sets of diffeomorphisms with trivial centralizers.

We develop two approaches for achieving trivial centralizer for an open neighbourhood of Anosov diffeomorphisms near a hyperbolic toral automorphism $A$. One approach, detailed in Section 4, is through the fixed points of the elements of ${\mathcal U}_1(A)$. The other, detailed in Section 5, is through the periodic points of period $p$ of the elements of ${\mathcal U}_p(A)$ for an odd prime $p$. Both approaches require that $C(A) = \langle A\rangle \times \langle J\rangle$ and that the order of $J$ be a power of two. When these holds, we prove that when the number of fixed points of $A$ is bigger than $2^n$, then every $f\in{\mathcal U}_1(A)$ has trivial centralizer, or when the number of periodic points of $A$ of period an odd prime $p$ minus the number of fixed points of $A$ is bigger than $2^n$, then every $f\in{\mathcal U}_p(A)$ has trivial centralizer.

To complete the proof of the Main Theorem we exhibit in Section 6 the existence of an irreducible hyperbolic toral automorphism in each of $n=2$, $3$, and $4$, to which we can apply either the fixed point approach or the odd prime periodic point approach. For simplicity, each such example is constructed as a companion matrix of a monic irreducible polynomial whose constant term is $\pm 1$. However, for any toral automorphism conjugate to one of the given three examples, the needed conditions will be satisfied to obtain a nearby open neighbourhood of Anosov diffeomorphisms with trivial centralizer.

We conclude the paper in Section 7 will several examples illustrating the versatility and limitations of the two approaches. We demonstrate the  algebraic techniques used to determine if we can apply either approach to a given irreducible hyperbolic toral automorphism. Not all of the examples start with a companion matrix for an irreducible polynomial. We also give examples where we cannot apply the fixed point approach nor the odd prime periodic point approach because $C(A)\ne \langle A\rangle\times\langle J\rangle$ or the order of $J$ is not a power of two.

\subsection{Open questions}

Combining the nonempty interior of ${\mathcal T}^1({\mathbb T}^n)$ for $n=2,3,4$ with the empty interior of ${\mathcal T}^1(S^1)$ naturally leads to the following modification of Question \ref{empty}.

\begin{que} For what compact manifolds $M$ with $\mathrm{dim}(M)\geq 2$, does $\mathcal{T}^1(M)$ have nonempty interior?
\end{que}

For a compact manifold $M$ with $\mathrm{dim}(M)\leq 2$ the proof in~\cite{BCVW} uses Morse-Smale diffeomorphisms to produce an open set of diffeomorphisms containing a dense set with nontrivial centralizers.  A natural question is then what happens if a $C^1$ diffeomorphism has a nontrivial homoclinic class.

\begin{que} Let $M$ be a compact surface and $\mathcal{O}$ be the open set of $C^1$ Axiom A diffeomorphisms of $M$ with no-cycles that contain a nontrivial homoclinic class. Is there an open and dense set in $\mathcal{O}$ with trivial centralizer?
\end{que}

In higher dimensions the proof of nontrivial centralizers in~\cite{BCVW} uses what are called wild diffeomorphisms.  These can be characterized as $C^1$ generic diffeomorphisms with an infinite number of chain recurrent classes.  If there are a finite number of chain recurrent classes in higher dimension, then it may be possible to produce an open set with a trivial centralizer.

\begin{que} Let $M$ be a compact manifold with $\mathrm{dim}(M)\geq 3$.  Let $\mathcal{O}$ be the open set of $C^1$ Axiom A diffeomorphisms of $M$ with no-cycles.  Is there an open and dense set in $\mathcal{O}$ with trivial centralizer?
\end{que}

\section{Centralizers of Irreducible HTAs}

We identify the diffeomorphisms that commute with a hyperbolic toral automorphism. A classical rigidity result implies that any homeomorphism that commutes with an ergodic toral automorphism is affine (see Theorem \ref{c.matrix} below). An automorphism of ${\mathbb T}^n$ is  induced by a ${\rm GL}(n,{\mathbb Z})$ matrix $A$, i.e., an $n\times n$ integer matrix with determinant $\pm 1$. By abuse of notation, we will denote the induced toral automorphism by $A$ as well. A toral automorphism $A$ is ergodic with respect to the Lebesgue measure $\mu$ on ${\mathbb T}^n$ induced from ${\mathbb R}^n$ when $\mu(D)=0$ or $\mu({\mathbb T}^n\setminus D)=0$ for every measurable $A$-invariant $D\subset{\mathbb T}^n$. We denote an affine map of ${\mathbb T}^n$ by $B+c$ for $B\in {\rm GL}(n,{\mathbb Z})$ and a constant $c\in{\mathbb T}^n$, where $(B+c)(\theta) = B\theta +c$ for $\theta\in{\mathbb T}^n$. The centralizer of $A\in {\rm GL}(n,{\mathbb Z})$ within ${\rm GL}(n,{\mathbb Z})$ is the group
\[ C(A) = \{ B\in {\rm GL}(n,{\mathbb Z}): BA=AB\}.\]
We let ${\rm Per}^1(A)$ denote the fixed points of $A$. 

\vspace{0.1in}
\begin{thm}\label{c.matrix} $($Adler and Palais \cite{AP}$)$
Let $A\in{\rm GL}(n,{\mathbb Z})$ be ergodic with respect to $\mu$. If $h$ is a homeomorphism of $\mathbb{T}^n$ such that $hA =Ah$, then there exists $B\in C(A)$ and $c\in {\mathrm Per}^1(A)$ such that $h=B+c$.
\end{thm}

This rigidity gives an identification of $Z(A)$ for a hyperbolic toral automorphism $A$. An $A\in{\rm GL}(n,{\mathbb Z})$ having no eigenvalues of modulus one is hyperbolic. Since hyperbolicity implies ergodicity for $A$ (see \cite{Rob}), we can apply Theorem \ref{c.matrix} to show that any homeomorphism that commutes with $A$ is of the form $B+c$ for $B\in C(A)$ and $c\in{\rm Per}^1(A)$. On the other hand, for $B\in C(A)$ and $c\in{\rm Per}^1(A)$, it is straight-forward to show that $A$ and $B+c$ commute, i.e., that $B+c\in Z(A)$. Thus for a hyperbolic $A$ we have
\[ Z(A) = \{ B+c : B \in C(A), c\in {\rm Per}^1(A)\}.\]
The group $C(A)$ is countable because ${\rm GL}(n,{\mathbb Z})$ is countable. When $A$ is hyperbolic, the group ${\rm Per}^1(A)$ is finite, and so $Z(A)$ is countable as well. Since $-I\in C(A)$ we have that $-I\in Z(A)$ where $-I$ is not an integer power of a hyperbolic $A$. Thus, every hyperbolic toral automorphism has a nontrivial but countable centralizer.

The group structure of $C(A)$ is known when $A\in {\rm GL}(n,{\mathbb Z})$ is simple, i.e., $A$ has no repeated roots (see \cite{BR}). In particular, $C(A)$ is abelian when $A$ is simple. We describe this group structure when $A$ is irreducible, i.e., when its characteristic polynomial $p_A(x)$ is irreducible, which implies that $A$ is simple.  For $\lambda$ a root of the $p_A(x)$, the algebraic number field
\[ {\mathbb F}={\mathbb Q}(\lambda) = \{ r_1 + r_2\lambda + r_3\lambda^2+\cdot\cdot\cdot+ r_n\lambda^{n-1}:r_i\in{\mathbb Q} \} \]
has degree $n$ over ${\mathbb Q}$. The signature of $p_A(x)$ is the pair $(r_1,r_2)$ where $r_1$ is the number of real roots of $p_A(x)$, and $2r_2$ is the number of complex roots of $p_A(x)$. The elements of ${\mathbb F}$ that are roots of monic polynomials with integer coefficients form the ring of integers ${\mathfrak o}_{\mathbb F}$ in ${\mathbb F}$. Let ${\mathfrak o}_{\mathbb F}^\times$ denote the group of units in ${\mathfrak o}_{\mathbb F}$, i.e., those elements of ${\mathfrak o}_{\mathbb F}$ whose multiplicative inverses are also in ${\mathfrak o}_{\mathbb F}$. By Dirichlet's Unit Theorem (see \cite{Coh,Swi}), the group ${\mathfrak o}_{\mathbb F}^{\times}$ is isomorphic to the product of ${\mathbb Z}^{r_1+r_2-1}$ with a finite cyclic group generated by a root of unity $\zeta_{\mathbb F}$ of even order. The following group structure result for $C(A)$, with $A$ irreducible, is from Baake and Roberts \cite{BR} (see also \cite{KKS}). 

\begin{thm}\label{C(A)structure} If $A\in {\rm GL}(n,{\mathbb Z})$ is irreducible, and ${\mathbb F} = {\mathbb Q}(\lambda)$ for a root $\lambda$ of $p_A(x)$, then $C(A)$ is isomorphic to a subgroup of finite index of ${\mathfrak o}_{\mathbb F}^\times$.
\end{thm}

We may obtain a complete identification of $C(A)$ for an irreducible $A\in {\rm GL}(n,{\mathbb Z})$ by an examination of the proof of Theorem \ref{C(A)structure}. Any matrix with rational entries that commutes with $A$ belongs to the ring
\[ {\mathbb Q}[A]=\{r_1I+r_2A+r_3A^2+\cdot\cdot\cdot+r_n A^{n-1}:r_i\in{\mathbb Q} \}\]
(see \cite{Jac}). For a root $\lambda$ of $p_A(x)$, there is a ring isomorphism $\gamma$ from ${\mathbb Q}[A]$ to ${\mathbb F}={\mathbb Q}(\lambda)$ given by
\[ \gamma:v(A)\to v(\lambda)\]
for $v$ in the polynomial ring ${\mathbb Q}[x]$. If $B\in {\mathbb Q}[A]$ has integer entries, then $\gamma(B)\in{\mathfrak o}_{\mathbb F}$, and if $B\in{\mathbb Q}[A]\cap {\rm GL}(n,{\mathbb Z})$, then $\gamma(B)\in{\mathfrak o}_{\mathbb F}^\times$. (The converse of each of these is false; see \cite{Bak1} or Example \ref{noninteger} in this paper for a counterexample.) Since $C(A) = {\mathbb Q}[A]\cap {\rm GL}(n,{\mathbb Z})$, we have that $\gamma( C(A) )\subset {\mathfrak o}_{\mathbb F}^\times$. On the other hand, the ring
\[ {\mathbb Z}[A] = \{ m_1 I+m_2A+m_3A^2+\cdot\cdot\cdot+m_nA^{n-1}:m_i\in {\mathbb Z}\}\]
is a subring of  ${\mathbb Q}[A]$, and the image of ${\mathbb Z}[A]$ under $\gamma$ is the ring
\[ {\mathbb Z}[\lambda] = \{ m_1 + m_2\lambda+m_3\lambda^3+\cdot\cdot\cdot + m_n\lambda^{n-1}: m_i\in{\mathbb Z}\}.\]
If $B\in {\mathbb Z}[A]\cap {\rm GL}(n,{\mathbb Z})$, then $\gamma(B)\in {\mathbb Z}[\lambda]^\times$, the group of units in ${\mathbb Z}[\lambda]$. Since ${\mathbb Z}[A]\cap {\rm GL}(n,{\mathbb Z})\subset C(A)$, we get ${\mathbb Z}[\lambda]^\times\subset \gamma(C(A))\subset {\mathfrak o}_{\mathbb F}^\times$. Since ${\mathbb Z}[\lambda]^\times$ is a finite index subgroup of ${\mathfrak o}_{\mathbb F}^\times$, we obtain  that $\gamma(C(A))$ is a finite index subgroup of ${\mathfrak o}_{\mathbb F}^\times$. A simple squeeze play on ${\mathbb Z}[\lambda]^\times$ and ${\mathfrak o}_{\mathbb F}^\times$ gives the proof of the following result that completely classifies $C(A)$ for certain irreducible toral automorphisms $A$.

\begin{thm}\label{IdentificationC(A)} Suppose $A\in {\rm GL}(n,{\mathbb Z})$ is irreducible. Let $\lambda$ be a root of $p_A(x)$ and ${\mathbb F}={\mathbb Q}(\lambda)$. If ${\mathbb Z}[\lambda]^\times = {\mathfrak o}_{\mathbb F}^\times$, then $\gamma(C(A)) = {\mathfrak o}_{\mathbb F}^\times$.
\end{thm}

The condition ${\mathbb Z}[\lambda]^\times = {\mathfrak o}_{\mathbb F}^\times$ of Theorem \ref{IdentificationC(A)} holds for those irreducible toral automorphisms $A$ for which ${\mathbb Z}[\lambda] = {\mathfrak o}_{\mathbb F}$. Not every irreducible $A$ has ${\mathbb Z}[\lambda]={\mathfrak o}_{\mathbb F}$, for $\lambda$ a root of $p_A(x)$ and ${\mathbb F} = {\mathbb Q}(\lambda)$, but such $A$ do exist (see Section 6). It can also happen that ${\mathbb Z}[\lambda]^\times = {\mathfrak o}_{\mathbb F}^\times$ while ${\mathbb Z}[\lambda] \ne {\mathfrak o}_{\mathbb F}$ (see Example \ref{not} in this paper). But as long as $\gamma(C(A)) = {\mathfrak o}_{\mathbb F}^\times$, there is $J\in C(A)$ such that $\gamma(J) = \zeta_{\mathbb F}$. The element $-I$ that always belongs to $C(A)$ satisfies $\gamma(-I) = -1$. Because $-I$ has order two, it follows that $-I\in\langle J\rangle$, so that the order of $J$ is always even. Some basic properties of the finite cyclic group $\langle J\rangle$ are listed in the following results, where we use basic cyclic group theory \cite{Hun} without explicit reference.

\begin{cla}\label{even} Suppose $J\in{\rm GL}(n,{\mathbb Z})$ has order $2k$ for some $k\in{\mathbb N}$. If $-I\in\langle J\rangle$, then $J^k=-I$.
\end{cla}

\begin{proof} With $-I\in\langle J\rangle$ and $\langle J\rangle$ cyclic of order $2k$, there is $l\in{\mathbb Z}$ with $0<l<2k$ such that $J^l = -I$. Then $J^{2l} = (-I)^2 = I$. Hence $2k \mid 2l$. Thus there is $m\in{\mathbb Z}$ such that $2l = 2mk$, or $l=mk$. Suppose $m=2s$ for some $s\in{\mathbb Z}$. Then $l=2sk$, and so
\[ -I = J^l = J^{2sk} = (J^{2k})^s = I^s = I,\]
a contradiction. So $m=2s+1$ for some $s\in{\mathbb Z}$. Then
\[ J^l = J^{(2s+1)k} = J^{2sk} J^k = (J^{2k})^sJ^k = I^sJ^k = J^k.\]
Since $J^l=-I$, we obtain $J^k = -I$.
\end{proof}

\begin{lem}\label{poweroftwo} Suppose that $J\in{\rm GL}(n,{\mathbb Z})$ has even finite order and that $-I\in\langle J\rangle$. If the order of $J$ is $2^b$ for some $b\in{\mathbb N}$, then for each $0<l< 2^b$, there exists $t\in{\mathbb N}$ such that $(J^l)^t = -I$. If the order of $J$ is $2^dk$ for $d\in{\mathbb N}$ and $k>2$ an odd integer, then for $l=2^d$ there is no $t\in{\mathbb N}$ for which $(J^l)^t=-I$.
\end{lem}

\begin{proof} Suppose that the order of $J$ is $2^b$ for some $b\in {\mathbb N}$. By Claim \ref{even}, we have that $J^{2^{b-1}}=J^{2^b/2}=-I$. For $t\in{\mathbb N}$ we have $J^{lt}=-I=J^{2^{b-1}}$ if and only if $lt \equiv 2^{b-1} {\rm\ mod\ }2^b$.

The integer $l$ between $0$ and $2^b$ is odd or even. Suppose $l=2m+1$ for some $m\in{\mathbb N}$. Then $t=2^{b-1}$ satisfies 
$$lt = (2m+1)2^{b-1} = 2^{b-1} + m2^b,$$ 
so that $lt\equiv 2^{b-1} {\rm mod\ }2^b$. Hence $J^{lt} = -I$ when $l$ is odd and $t=2^{b-1}$.

Suppose that $l=2r_1$ for some $r_1\in{\mathbb N}$. Since $l<2^b$, we have that $r_1<2^{b-1}$. If $r_1=2m+1$ for some $m\in{\mathbb N}$, then $t=2^{b-2}$ satisfies 
$$lt = 2(2m+1)2^{b-2} = 2^{b-1}+m2^b,$$ 
so that $J^{lt}=-I$. Otherwise, $r_1=2r_2$ with $r_2\in{\mathbb N}$ and $r_2< 2^{b-2}$ because $r_1<2^{b-1}$. If $r_2=2m+1$ for some $m\in{\mathbb N}$, then $t=2^{b-3}$ satisfies 
$$lt = 4(1+2m)2^{b-3} = 2^{b-1}+m2^b,$$ 
so that $J^{lt}=-I$. Continuing this  process gives $t=2^{b-u}$ for some $u\in{\mathbb N}$ with $u\leq b$ for which $lt = 2^{b-1}+m2^b$, and hence $J^{lt} = -I$.

Now suppose that the order of $J$ is $2^dk$ for some $d\in{\mathbb N}$ and $k>2$ an odd integer. For $l=2^d$, suppose there is $t\in{\mathbb N}$ such that $(J^l)^t=-I$. Since $J^{2^{d-1}k}=-I$ by Claim \ref{even}, then $lt \equiv 2^{d-1}k {\rm\ mod\ }2^dk$. So we have that $2^dt = 2^{d-1}k + m2^dk$ for some $m\in{\mathbb Z}$. This implies that $2(t-mk)=k$. This contradicts the oddness of $k$. So there is no integer power of $J^l$ which equals $-I$. 
 \end{proof}
 
The irreducible toral automorphisms $A$ of interest here are those for which the abelian $C(A)$ has one infinite cyclic factor. For $\lambda$ a root of $p_A(x)$ and ${\mathbb F}={\mathbb Q}(\lambda)$, we know by Theorem \ref{C(A)structure} that $\gamma(C(A))$ is a finite index subgroup of ${\mathfrak o}_{\mathbb F}^\times$. The condition of $C(A)$ of having rank one is then the same as ${\mathfrak o}_{\mathbb F}^\times$ having rank one. The latter has rank one when there exists a fundamental unit $\epsilon_{\mathbb F}$ such that
\[ {\mathfrak o}_{\mathbb F}^\times = \langle \epsilon_{\mathbb F}\rangle \times \langle \zeta_{\mathbb F}\rangle.\]
This happens precisely when $p_A(x)$ has signature $(r_1,r_2)$ satisfying
\[ r_1+r_2-1=1.\]
This does not happen when $n\geq 5$ because all of the possibilities for $r_1$ and $r_2$ imply that $r_1+r_2-1\geq 2$. Only when $n=2,3,4$ can we satisfy $r_1+r_2-1=1$. For $n=2$ this requires that $r_1=2$, $r_2=0$ (a real quadratic field), for $n=3$ this requires that $r_1=1$, $r_2=1$ (a complex cubic field), and for $n=4$ this requires that $r_1=0$, $r_2=2$ (a totally complex quartic field). 

\begin{cor}\label{cases234} Suppose $A\in{\rm GL}(n,{\mathbb Z})$ is irreducible whose characteristic polynomial $p_A(x)$ has signature $(r_1,r_2)$ satisfying $r_1+r_2-1=1$. For $\lambda$ a root of $p_A(x)$ and ${\mathbb F}={\mathbb Q}(\lambda)$, if ${\mathbb Z}[\lambda]^\times = {\mathfrak o}_{\mathbb F}^\times$ and $\lambda=\epsilon_{\mathbb F}$, then $C(A) = \langle A\rangle \times \langle J\rangle$ where $\gamma(A) = \epsilon_{\mathbb F}$ and $\gamma(J) = \zeta_{\mathbb F}$, and where $J$ satisfies $J^k=-I$ for $k$ half the order of $J$.
\end{cor}

\begin{proof} Suppose that ${\mathbb Z}[\lambda]^\times = {\mathfrak o}_{\mathbb F}^\times$ and $\lambda = \epsilon_{\mathbb F}$. Then by Theorem \ref{IdentificationC(A)} we have that $\gamma(C(A)) = {\mathfrak o}_{\mathbb F}^\times = \langle \epsilon_{\mathbb F}\rangle \times \langle \zeta_{\mathbb F}\rangle$. Since $\gamma(A) = \lambda=\epsilon_{\mathbb F}$ and since $\gamma(J) = \zeta_{\mathbb F}$, we obtain
\[ C(A) = \langle A \rangle \times \langle J\rangle.\]
Since $-I\in \langle J\rangle$, and $J$ has order $2k$ for some $k\in{\mathbb N}$, we have by Claim \ref{even} that $J^k = -I$.
\end{proof}

The possibilities for the order of $J$ are known when $C(A)=\langle A\rangle \times \langle J\rangle$ for $\gamma(A)= \epsilon_{\mathbb F}$ and $\gamma(J) = \zeta_{\mathbb F}$. Each real quadratic or complex cubic field ${\mathbb F}$ has a real embedding, implying that the generator $\zeta_{\mathbb F}$ of the finite cyclic factor of ${\mathfrak o}_{\mathbb F}^\times$ is $-1$, hence $J=-I$ has order $2$. Each totally complex quartic field ${\mathbb F}$ has no real embedding, so it is possible for the generator $\zeta_{\mathbb F}$ of the finite cyclic factor of ${\mathfrak o}_{\mathbb F}^\times$ to be $-1$, or to be a complex root of unity of order $4$, $6$, $8$, $10$, or $12$ (see \cite{PZ}) where all of the possibilities do occur (see \cite{PS}). But only when the order of $J$ is $2$, $4$, or $8$, do we have by Lemma \ref{poweroftwo}, that for each $l$ with $0<l<2^b-1$ there exists $t\in{\mathbb N}$ such that $(J^l)^t=-I$. The existence of this $t$, when the order of $J$ is a power of two, and the element $-I\in C(A)$ play key roles in the fixed point approach and odd prime periodic point approach to proving the Main Theorem.

\section{$C^1$ Neighbourhoods of HTAs}

We use the structural stability of a hyperbolic toral automorphism to obtain an open set of diffeomorphisms in which to search for an open subset that may have trivial centralizer. A compact set $\Lambda\subset M$ invariant under $f\in{\rm Diff}^1(M)$ is hyperbolic if there exists a splitting of the  tangent space $T_{\Lambda}M=\mathbb{E}^u\oplus \mathbb{E}^s$ and positive constants $C$ and $\lambda<1$ such that, for any point $x\in\Lambda$ and any $n\in\mathbb{N}$,
$$
\begin{array}{llll}
\| D_xf^{n}v\|\leq C \lambda^{n}\| v\|,\textrm{ for }v\in
E^{s}_x \textrm{, and}\\
\| D_xf^{-n}v\|\leq C \lambda^{n}\| v\|,\textrm{ for }v\in
E^{u}_x. \end{array}
$$
When $M$ is hyperbolic for $f$, we say $f$ is an {\it Anosov diffeomorphism}. Every hyperbolic toral automorphism is Anosov. The next statement is a standard result, see for instance~\cite{KH}.

\begin{thm}\label{structural} Let $f\in \mathrm{Diff}^1(M)$ and $\Lambda$ be a hyperbolic set for $f$.  Then for any neighborhood $V$ of $\Lambda$ and every $\delta>0$ there exists a neighborhood $\mathcal{U}$ of $f$ in $\mathrm{Diff}^1(M)$ such that for any $g\in \mathcal{U}$ there is a hyperbolic set $\Lambda_g\subset V$ and a homeomorphism $h:\Lambda_g\rightarrow \Lambda$ with $d_{C^0}(\mathrm{id},h) + d_{C^0}(\mathrm{id},h^{-1})<\delta$ and $h\circ g|_{\Lambda_g}=f|_{\Lambda}\circ h$.  Moreover, $h$ is unique when $\delta$ is sufficiently small.
\end{thm}

Combining Theorems \ref{c.matrix} and \ref{structural} gives an affine representation, via topological conjugacy, of the elements of the centralizer of any Anosov diffeomorphism that is $C^1$ close to a hyperbolic toral automorphism.

\begin{thm}\label{nbhd1}
If $A\in{\rm GL}(n,{\mathbb Z})$ is hyperbolic, then there exists a neighborhood\, $\mathcal{U}(A)$ of $A$ in $\mathrm{Diff}^1(\mathbb{T}^n)$ such that for each $f\in\mathcal{U}(A)$ there exists a homeomorphism $h_f$ of ${\mathbb T}^n$ for which every $g\in Z(f)$ satisfies $h_f g h_f^{-1} = B+c$ for some $B\in C(A)$ and some $c\in {\mathrm Per}^1(A)$.
\end{thm}

\begin{proof}
By Theorem \ref{structural} there exists a neighborhood $\mathcal{U}(A)$ of $A$ in $\mathrm{Diff}^1(\mathbb{T}^n)$ such that each $f\in \mathcal{U}(A)$ is topologically conjugate to $A$ by a homeomorphism $h_f$ that is close to the identity:
\[ f = h_f^{-1}Ah_f.\]
For $g\in Z(f)$, we have that
\begin{align*}
h_f g h_f^{-1}A
& = h_f g(h_f^{-1}Ah_f)h_f^{-1} \\
&= h_f gf h_f^{-1} \\
& = h_f fg h_f^{-1} \\
& = (h_f f h_f^{-1})h_f g h_f^{-1} \\
& = A h_f g h_f^{-1}.
\end{align*}
Thus the homeomorphism $h_f g h_f^{-1}$ commutes with the hyperbolic toral automorphism $A$. Use of Theorem \ref{c.matrix} shows that the  homeomorphism $h_f gh_f^{-1}$ is affine, so that
\[ h_f g h_f^{-1} = B+c\]
for some $B\in C(A)$ and some $c\in{\rm Per}^1(A)$.
\end{proof}

The affine representation of $Z(f)$ in Theorem \ref{nbhd1} gives rise to a complete characterization of $Z(f)$ for every $f\in {\mathcal U}(A)$ when $A$ is hyperbolic. The map $g\to h_f g h_f^{-1}$ from the group of homeomorphisms of ${\mathbb T}^n$ to itself is an inner automorphism. Thus the group
\[ H_f(A) = h_f^{-1}(Z(A))h_f\] 
is isomorphic to $Z(A)$. It follows that because $Z(A)$ is countable, the group $H_f(A)$ is countable for each $f\in{\mathcal U}(A)$. As is well known \cite{Sma1}, Theorem \ref{structural} implies that each $H_f(A)$ is discrete, i.e., the only homeomorphism close to the identity that commutes with $f$ is the identity. The proof of the following is immediate by Theorem \ref{nbhd1}.

\begin{cor}\label{homeo} If $A\in {\rm GL}(n,{\mathbb Z})$ is hyperbolic, then for each $f\in {\mathcal U}(A)$ we have
\[ Z(f) = H_f(A)\cap {\rm Diff}^1({\mathbb T}^n).\]
\end{cor}

The centralizer of each $f\in {\mathcal U}(A)$ can be no bigger than $H_f(A)$ by Corollary \ref{homeo} when $A$ is hyperbolic. The group structure of $H_f(A)$ is the same for all $f\in {\mathcal U}(A)$, and is determined by the group structure of the $Z(A)$. The group structures of $C(A)$ and ${\rm Per}^1(A)$ determine the group structure of $Z(A)$. For instance, because $-I\in C(A)$, the homeomorphism
\[ g=h_f^{-1}(-I+c)h_f,\]
for any $c\in {\rm Per}^1(A)$, belongs to $H_f(A)$ and is an involution, i.e., $g^2={\rm id}$. Another example is the homeomorphism
\[ g=h_f^{-1}(-A+c)h_f,\] 
for $c\in {\rm Per}^1(A)$, that belongs to $H_f(A)$, but has infinite order. In both of these examples, neither of the given homeomorphisms in $H_f(A)$ is an integer power of $f$.

All of the diffeomorphisms $f\in {\mathcal U}(A)$ have the same dynamics in the topological sense because $h_f f h_f^{-1} =A$ when $A$ is hyperbolic. In particular, each $f\in {\mathcal U}(A)$ has the same number of periodic points of given period as that of $A$. For $k\in{\mathbb N}$, let ${\rm Per}^k(f)$ denote the set of periodic points of $f$ of periodic $k$. The set ${\rm Per}^k(A)$ is a finite subgroup of ${\mathbb T}^n$ when $A^k-I$ is invertible. In this case, by Pontryagin duality, ${\rm Per}^k(A)$ is isomorphic to the generalized Bowen-Franks group
\[ BF_k(A) = {\mathbb Z}^n/(A^k-I){\mathbb Z}^n\]
(see \cite{MR}). The groups $BF_k(A)$ play a significant role in the classification of hyperbolic toral automorphisms \cite{BM}, and in the classification of quasiperiodic flows of Koch type under projective conjugacy \cite{Bak2}. The group structure of $BF_k(A)$ is found by computing the Smith normal form of $A^k-I$ (see \cite{Coh}). From this we get the cardinality $\vert BF_k(A)\vert$, and hence the cardinality $\vert {\rm Per}^k(A)\vert$. When $A$ is hyperbolic, then $A^k-I$ is invertible for all $k\in{\mathbb N}$, so that $\vert {\rm Per}^k(A)\vert$ is finite for all $k\in{\mathbb N}$, and so for each $f\in {\mathcal U}(A)$ the group ${\rm Per}^k(f)$ is finite for all $k\in{\mathbb N}$ as well.

The trivial centralizer problem for $f\in {\mathcal U}(A)$ reduces by Corollary \ref{homeo} to showing that the only diffeomorphisms in $H_f(A)$ are the integer powers of $f$. For a hyperbolic toral automorphism $A$, each affine map $B+c\in Z(A)$ corresponds to a homeomorphism $g = h_f^{-1}(B+c)h_f \in H_f(A)$ that commutes with $f$. The map $g$ permutes the points of the finite set ${\rm Per}^k(f)$ because for $p\in\mathrm{Per}^k(f)$ we have $g(p)\in\mathrm{Per}^k(f)$. If  $g$ is a diffeomorphism, then $g\in Z(f)$ and for $p\in\mathrm{Per}^k(f)$ we have
\[ D_{g(p)}f^kD_pg=D_pgD_pf^k,\]
implying that the linear maps $D_{g(p)}f^k$ and $D_pf^k$ are similar. However, because ${\rm Per}^k(A)$ is a finite set, there exists inside ${\mathcal U}(A)\setminus\{A\}$ an open subset ${\mathcal U}_k(A)$ in which any two distinct $k$-periodic orbits have different spectra (see \cite{KH}).

This sets the stage for proving the triviality of $Z(f)$ for $f$ in the open set ${\mathcal U}_k(A)$ when $A$ is a hyperbolic toral automorphism. Algebraic conditions on $A$ to achieve a trivial centralizer for $f$ are that $A$ is irreducible and
\[ C(A) = \langle A\rangle\times\langle J\rangle\]
where $J\in C(A)$ has order $2^b$ for some $b\in{\mathbb N}$. This can only happen when $n=2,3,4$. In this case, by Theorem \ref{nbhd1}, each $g\in Z(f)$ satisfies
\[ h_f g h_f^{-1} = A^m J^l + c\]
for some $m\in{\mathbb Z}$, some $0\leq l\leq 2^b-1$, and some $c\in {\rm Per}^1(A)$. We will show that there are readily verifiable numerical conditions on ${\rm Per}^p(A)$, for $p=1$ and/or for $p$ an odd prime, that force $l=0$ and $c=0$ for every $g\in Z(f)$, to give $g=f^m$.

\section{Fixed Point Approach}

An approach to obtaining the triviality of the centralizer of each $f\in {\mathcal U}_1(A)$ is through its fixed points. At issue are  the homeomorphisms in $H_f(A)\setminus \langle f\rangle$, and showing that they are not diffeomorphisms. One example of such a homeomorphism is the involution $g=h_f^{-1} (-I+c)h_f$ for each $c\in{\rm Per}^1(A)$. Through the Smith normal form, the values of $\vert{\rm Per}^1(A)\vert$ and $\vert {\rm Per}^1(-I)\vert$ are readily computable. For the latter we have $\vert {\rm Per}^1(-I)\vert = 2^n$ when $I$ is $n\times n$.

\begin{lem}\label{minusidentity} Suppose $A\in{\rm GL}(n,{\mathbb Z})$ is hyperbolic.  If
\[ \vert{\rm Per}^1(A)\vert>\vert{\rm Per}^1(-I)\vert,\]
then for each $f\in{\mathcal U}_1(A)$, no $g\in Z(f)$ satisfies $h_fgh_f^{-1}=-I+c$ where $c\in{\rm Per}^1(A)$.
\end{lem}

\begin{proof} Let $f\in {\mathcal U}_1(A)$ be arbitrary. Label the finitely many elements of ${\rm Per}^1(A)$ (the fixed points of $A$) by
\[ c_0=0, c_1,c_2,\dots, c_{k-1}.\]
Then the fixed points of $f$ are given by
\[ \theta_i = h_f^{-1}(c_i),\ i=0,1,\dots, k-1.\]
For $g\in Z(f)$, we have by Theorem \ref{nbhd1} that $h_f g h_f^{-1} = B+c$ where $B\in C(A)$ and $c\in {\rm Per}^1(A)$. Suppose $B=-I$.

Suppose that $c=0$. Then we have that
\[ 0 = -I(0) = h_f gh_f^{-1}(0) = h_f(g(\theta_0)),\]
so that
\[ \theta_0 = h_f(0) = g(\theta_0).\]
By hypothesis, $-I$ has less than $k$ fixed points, and so $g$ (being topologically conjugate to $-I$) also has less than $k$ fixed points. There is then $1\leq j\leq k-1$ such that $g(\theta_j)\ne \theta_j$. Since $g$ permutes the fixed points of $f$, there is $1\leq l\leq k-1$, $l\ne j$ , such that $g(\theta_j) = \theta_l$. From $g\in Z(f)$, i.e., $fg=gf$ with $g$ a diffeomorphism, we then get
\[ D_{\theta_l} f D_{\theta_j} g = D_{g(\theta_j)}f D_{\theta_j} g = D_{f(\theta_j)} g D_{\theta_j} f =D_{\theta_j} g D_{\theta_j} f.\]
This says that $D_{\theta_l} f$ and $D_{\theta_j} f$ are similar, contradicting that $f\in{\mathcal U}_1(A)$. This implies that no $g\in Z(f)$ satisfies $h_f g h_f^{-1} = -I$.

Suppose that $c\ne 0$. Then we have that $c = c_i$ for some $1\leq i\leq k-1$. It follows that
\[ c_i = (-I + c_i)(0) = h_f g h_f^{-1} (0)  = h_f (g(\theta_0)).\]
This implies that
\[ \theta_i = h_f^{-1}(c_i) = g(\theta_0).\]
From $g\in Z(f)$ we get
\[ D_{\theta_i} f D_{\theta_0} g =D_{g(\theta_0)} f D_{\theta_0}g  = D_{f(\theta_0)} g D_{\theta_0}f = D_{\theta_0} g D_{\theta_0} f.\]
This says that $D_{\theta_i} f$ and $D_{\theta_0}f $ are similar, contradicting that $f\in {\mathcal U}_1(A)$. So no $g\in Z(f)$ satisfies $h_f g h_f^{-1} = -I +c$.
\end{proof}

Another type of  homeomorphism in $H_f(A)$ is the finite order $g=h_f^{-1}(I+c)h_f$ for a nonzero $c\in{\rm Per}^1(A)$, when such a $c$ exists.  Each such $g$ has finite order because ${\rm Per}^1(A)$ is finite. We can eliminate these as diffeomorphism without any conditions on ${\rm Per}^1(A)$.

\noindent\begin{lem}\label{identity} Suppose $A\in{\rm GL}(n,{\mathbb Z})$ is hyperbolic. Then for each $f\in {\mathcal U}_1(A)$, no $g\in Z(f)$ satisfies $h_f g h_f^{-1} = I+c$ for nonzero $c\in{\rm Per}^1(A)$.
\end{lem}

\begin{proof} Let $f\in {\mathcal U}_1(A)$ be arbitrary. Let $c_0=0,c_1,\dots,c_{k-1}$ be the fixed points of $A$, and $\theta_i=h_f^{-1}(c_i)$, $i=0,\dots,k-1$, the fixed points of $f$. Suppose $h_f g h_f^{-1}=I+c$ for $c$ a nonzero fixed point of $A$. Then $c=c_i$ for some $1\leq i\leq k-1$, so that
\[ h_f(\theta_i) = c_i  = (I+c_i)(0) = h_f gh_f^{-1}(0) = h_f(g(\theta_0)).\]
This says that $\theta_i  = g (\theta_0)$. From this and $g\in Z(f)$, we get
\[ D_{\theta_i} f D_{\theta_0} g = D_{g(\theta_0)} f D_{\theta_0}g  = D_{f(\theta_0)} g D_{\theta_0}f = D_{\theta_0} g D_{\theta_0} f.\]
This says that $D_{\theta_i}f$ and $D_{\theta_0} f$ are similar, contradicting that $f\in{\mathcal U}_1(A)$. So no $g\in Z(f)$ satisfies $h_f g h_f^{-1} = I + c$ for a nonzero $c\in {\rm Per}^1(A)$.
\end{proof}

A homeomorphism in $H_f(A)$ of infinite order is $g= h_f^{-1}(A^m+c)h_f$, for $m\in{\mathbb Z}$ and a nonzero $c\in{\rm Per}^1(A)$, when such a $c$ exists. We can now eliminate these as diffeomorphisms without conditions on ${\rm Per}^1(A)$.

\begin{lem}\label{plusminus} Suppose $A\in{\rm GL}(n,{\mathbb Z})$ is hyperbolic. For each $f\in {\mathcal U}_1(A)$, no $g\in Z(f)$ satisfies $h_f g h_f^{-1}= A^{m} + c$ for $m\in{\mathbb Z}$ and nonzero $c\in{\rm Per}^1(A)$.
\end{lem}

\begin{proof} For $g\in Z(f)$ suppose that $h_f g h_f^{-1} = A^m + c$ for $m\in{\mathbb Z}$ and $c$ a nonzero fixed point of $A$. Since $h_f f  h_f^{-1} = A$, we have that $h_f f^{-m}h_f^{-1}=A^{-m}$. Since $g\in Z(f)$ and $f^{-m}\in Z(f)$, we have that $f^{-m}g\in Z(f)$. Then
\begin{align*}
h_f f^{-m}g h_f^{-1}
& = (h_f f^{-m} h_f^{-1}) ( h_f g h_f^{-1}) \\
& = A^{-m}(A^{m} + c) \\
& = I + A^{-m}c \\
& = I + c,
\end{align*}
since $c\in{\rm Per}^1(A)$. But with $f\in {\mathcal U}_1(A)$, it is impossible for $f^{-m}g\in Z(f)$ to satisfy $h_f (f^{-m}g)h_f^{-1}= I + c$ for $c\ne 0$ by Lemma \ref{identity}.
\end{proof}

Having shown that several homeomorphisms in $H_f(A)$ are not diffeomorphisms through the fixed point approach, we are now in a position to state the sufficient conditions by which each $f\in {\mathcal U}_1(A)$ has trivial centralizer. 

\begin{thm}\label{precursorfixed} Let $A\in{\rm GL}(n,{\mathbb Z})$ be hyperbolic. If $C(A) = \langle A\rangle \times \langle J\rangle$ where the order of $J$ is $2^b$ for some $b\in{\mathbb N}$, and $\vert {\rm Per}^1(A)\vert>\vert {\rm Per}^1(-I)\vert$, then $Z(f)$ is trivial for all $f\in {\mathcal U}_1(A)$.
\end{thm}

\begin{proof} Let $f\in {\mathcal U}_1(A)$ for hyperbolic $A$, and let $g\in Z(f)$. By Theorem \ref{nbhd1} we have that
\[ h_f g h_f^{-1} = B+c\]
where $B\in C(A)$ and $c\in {\rm Per}^1(A)$. With $C(A) = \langle A\rangle \times \langle J\rangle$ where the order of $J$ is $2^b$ for some $b\in{\mathbb N}$, we have that $B=A^mJ^l$ for some $m\in {\mathbb Z}$ and some $0\leq l\leq 2^b-1$. Hence
\[ h_f g h_f^{-1} = A^m J^l +c.\]

Suppose that $l\ne 0$. By Lemma \ref{poweroftwo}, there is $t\in{\mathbb N}$ such that $(J^l)^t=-I$. This implies that
\[ h_f g^t h_f^{-1} = -A^{mt}+\tilde c\]
for some $\tilde c\in {\rm Per}^1(A)$ because $J\in Z(A)$ permutes the elements of ${\rm Per}^1(A)$. Since $h_f f^{-mt} h_f^{-1} = A^{-mt}$, we have
\[ h_f f^{-mt}g^t h_f^{-1} = -I + \tilde c.\]
This is impossible by Lemma \ref{minusidentity} because $\vert{\rm Per}^1(A)\vert > \vert{\rm Per}^1(-I)\vert$.

So it must be that $l=0$. Then we have $h_f g h_f^{-1} = A^m+c$. By Lemma \ref{plusminus} it must be that $c=0$, so that $h_f g h_f^{-1} = A^m$. Since $h_f f^m h_f^{-1}=A^m$, we obtain $g=f^m$. With $g\in Z(f)$ being arbitrary we have that $Z(f)$ is trivial. Since $f$ is arbitrary, we have that $Z(f)$ is trivial for all $f\in {\mathcal U}_1(A)$.
\end{proof}

We explain the necessity in Theorem \ref{precursorfixed} of the order of $J$ being $2^b$ for some $b\in{\mathbb N}$. For a hyperbolic $A\in{\rm GL}(n,{\mathbb Z})$ suppose that $C(A) = \langle A \rangle \times \langle J\rangle$ where the order of $J$ is $2^dk$ for some $d\in{\mathbb N}$ and some odd integer $k>2$. By Lemma \ref{poweroftwo}, for $l=2^d$ there is no $t\in {\mathbb N}$ such that $(J^l)^t=-I$. For each $f\in {\mathcal U}(A)$, an element of $H_f(A)$ is
\[ g= h_f^{-1}(J^l)h_f.\]
Since there is no $t\in {\mathbb N}$ for which $(J^l)^t=-I$, there is no $t\in{\mathbb N}$ for which $h_f g^th_f^{-1}=-I$. Thus, if we suppose that $g\in Z(f)$, we cannot apply Lemma \ref{minusidentity} to get a contradiction. Since $(J^l)^k=I$, we have that $h_f g^k h_f^{-1} = (J^l)^k = I$. If we suppose that $g\in Z(f)$, we cannot apply Lemma \ref{identity} to get a contradiction because $c=0$. The order of $g$ is finite because $g^k = {\rm id}$. It follows that $g$ cannot be a power of $f$ because the order of the hyperbolic $A= h_f f h_f^{-1}$ is infinite. We are left with the possibility of $Z(f)$ containing an element that is not a power of $f$.

\section{Odd Prime Periodic Point Approach}

An approach to obtaining the triviality of the centralizer of each $f\in {\mathcal U}_p(A)$ is through its periodic points of period an odd prime $p$. As with the fixed point approach, at issue is showing that none of the homeomorphisms in $H_f(A)\setminus\langle f\rangle$ are diffeomorphisms, wherein the involutions $g=h_f^{-1}(-I+c)h_f$, $c\in {\rm Per}^1(A)$, play a key role. The necessity of $p$ being an odd prime in this approach is seen in the proof of the following result about involutions commuting with $f$, a result that is presented in the context of an arbitrary nonempty set $S$.

\begin{lem}\label{involution2} For a bijection $f:S\to S$, suppose that a bijection $g:S\to S$ satisfies $fg=gf$, $g^2={\rm id}$, and $g\ne {\rm id}$. If for a odd prime $p$ there exists $\theta\in {\rm Per}^p(f) - {\rm Per}^1(f)$ such that $g(\theta)\ne \theta$, then
\[ g(\theta)\not\in\{ f^k(\theta):k=0,\dots,p-1\}.\]
\end{lem}

\begin{proof} By hypothesis there is $\theta\in{\rm Per}^p(f)-{\rm Per}^1(f)$ such that $g(\theta)\ne \theta$. Suppose that there is $k\in\{1,2,\dots,p-1\}$ such that $g(\theta) = f^k(\theta)$. Then for all $j\in{\mathbb Z}$, there holds
\[ g(f^j(\theta)) = f^j(g(\theta)) = f^j(f^k(\theta)) = f^{j+k}(\theta).\]
Since $g^2={\rm id}$, we have that
\[f(\theta) = g^2(f(\theta)) = g(g(f(\theta)))  = g(f^{1+k}(\theta)) = f^{1+2k}(\theta).\]
This implies that $f^{2k}(\theta) = \theta$. Since $\theta \in {\rm Per}^p(f)-{\rm Per}^1(f)$ and $p$ is prime, we have $p\mid 2k$, or $2k=mp$ for some $m\in{\mathbb Z}$. This says that $2 \mid mp$. Because $p$ is odd, we have $2$ and $p$ are relatively prime. Thus we have by Euclid's Lemma that $2 \mid m$. Thus $m=2r$ for some $r\in {\mathbb Z}$. Hence $2k=2rp$ and so $k=rp$. But $1\leq k\leq p-1$, a contradiction. Thus $g(\theta)\ne f^k(\theta)$ for all $k\in\{0,1,2,\dots,p-1\}$.
\end{proof}

In the odd prime periodic point approach we will have need of the following result relating ${\rm Per}^1(-I)$ and ${\rm Per}^1(-I+c)$ for any $c\in{\mathbb T}^n$.

\begin{cla}\label{shift} For any $c\in{\mathbb T}^n$, there holds
\[ \vert {\rm Per}^1(-I)\vert \geq \vert {\rm Per}^1(-I+c)\vert.\]
\end{cla}

\begin{proof} Let $\tau\in {\rm Per}^1(-I+c)$. Then
\[ \tau = (-I+c)(\tau) = -\tau + c.\]
From this we have $-\tau = \tau - c$. Define $c/2$ in the obvious way. Then
\[ (-I)(\tau - c/2) = -\tau + c/2 = (\tau - c) + c/2  = \tau - c/2.\]
So $\tau - c/2\in {\rm Per}^1(-I)$. We have associated to each $\tau\in {\rm Per}^1(-I+c)$ the $\tau - c/2\in {\rm Per}^1(-I)$. This injection gives the inequality.
\end{proof}

For $g\in Z(f)$ with $f\in {\mathcal U}_p(A)$, we can eliminate $l\ne0$ in $h_f g h_f^{-1} = A^mJ^l+c$, when there are sufficiently many non-fixed periodic points of period an odd prime $p$, and the order of $J$ is a power of two with $-I\in\langle J\rangle$.

\begin{lem}\label{elimination} Let $A\in{\rm GL}(n,{\mathbb Z})$ be hyperbolic. If there is $J\in C(A)$ with the order of $J$ being $2^b$ for some $b\in{\mathbb N}$, if $-I\in\langle J\rangle$, and if
\[ \vert {\rm Per}^p(A) - {\rm Per}^1(A)\vert > \vert {\rm Per}^1(-I)\vert\]
for some odd prime $p$, then for every $f\in {\mathcal U}_p(A)$, no $g\in Z(f)$ satisfies $h_f gh_f^{-1} = A^m J^l + c$ for $m\in{\mathbb Z}$, $1\leq l\leq 2^b-1$, and $c\in {\rm Per}^1(A)$.
\end{lem}

\begin{proof} Let $f\in {\mathcal U}_p(A)$. Suppose $g\in Z(f)$ satisfies $h_f g h_f^{-1} = A^m J^l+c$ for $m\in{\mathbb Z}$, $1\leq l\leq 2^b-1$, and $c\in {\rm Per}^1(A)$. By Lemma \ref{poweroftwo} there exists $t\in{\mathbb N}$ such that $J^{lt} = -I$. Then
\[ h_f g^th_f^{-1} = (A^m J^l+c)^t = A^{mt}J^{lt}+\tilde c = - A^{mt} + \tilde c\]
where $\tilde c\in {\rm Per}^1(A)$ because $J\in C(A)$ permutes the fixed points of $A$. Since $h_f f h_f^{-1}=A$, then $h_f f^{-mt} h_f^{-1} = A^{-mt}$, so that
\[ h_f f^{-mt}g^t h_f^{-1} = (A^{-mt})(-A^{mt}+\tilde c) = -I + \tilde c.\]
The diffeomorphism $f^{-mt}g^t$, which belongs to $Z(f)$ because $f,g\in Z(f)$, is not the identity but is an involution:
\[ h_f (f^{-mt}g^t)^2h_f = (-I+\tilde c)^2 = I - \tilde c + \tilde c = I.\]

By hypothesis and Claim \ref{shift} we have that
\[ \vert {\rm Per}^p(A) - {\rm Per}^1(A)\vert > \vert {\rm Per}^1(-I)\vert \geq \vert {\rm Per}^1(-I+\tilde c)\vert.\]
This implies the existence of $d\in {\rm Per}^p(A)-{\rm Per}^1(A)$ such that $(-I+\tilde c)(d)\ne d$. The point  $(-I+\tilde c)(d)=-d+\tilde c$ belongs to ${\rm Per}^p(A)-{\rm Per}^1(A)$ because
\[ A^p(-d+\tilde c) = -A^p d + \tilde c = -d+\tilde c\]
and if $-d+\tilde c = A(-d+\tilde c) = -Ad + \tilde c$, then $d\in{\rm Per}^1(A)$, a contradiction. Setting $\theta=h_f^{-1}(d)$ gives $(f^{-mt}g^t)(\theta)\ne\theta$ where $\theta$ and $(f^{-mt}g^t)(\theta)$ both belong to ${\rm Per}^p(f) - {\rm Per}^1(f)$.

Since $f^{-mt}g^t\ne{\rm id}$, it follows by Lemma \ref{involution2} that
\[ (f^{-mt}g^t)(\theta) \not\in \{ \theta,f(\theta),\dots, f^{p-1}(\theta)\}.\]
This says the the $p$-periodic orbits of $(f^{-mt}g^t)(\theta)$ and $\theta$ under $f$ are distinct. Since $f^p(f^{-mt}g^t) = (f^{-mt}g^t) f^p$ we have that
\[ D_{f^{-mt}g^t(\theta)} f^p D_{\theta} (f^{-mt}g^t) = D_{f^p(\theta)} (f^{-mt}g^t) D_\theta f^p = D_\theta (f^{-mt}g^t) D_\theta f^p.\]
This says that $D_{f^{-mt}g^t(\theta)} f^p$ and $D_\theta f^p$ are similar, contradicting that $f\in {\mathcal U}_p(A)$.
\end{proof} 

It thus remains to show for $g\in Z(f)$ with $f\in {\mathcal U}_p(A)$ that $c=0$ in $h_f g h_f^{-1} = A^m + c$. This will be done for those homeomorphisms $g\in H_f(A)$ with $h_f g h_f^{-1}=I+c$. In this part of the the odd prime periodic point approach, we do not need $p$ odd or prime.

\begin{lem}\label{identity2} Let $A\in{\rm GL}(n,{\mathbb Z})$ be hyperbolic. If ${\rm Per}^k(A) - {\rm Per}^1(A)\ne\emptyset$ for some $k\geq 2$, then for every $f\in {\mathcal U}_k(A)$, no $g\in Z(f)$ satisfies $h_f g h_f^{-1} = I+c$ for nonzero $c\in {\rm Per}^1(A)$.
\end{lem}

\begin{proof} Suppose $g\in Z(f)$ satisfies $h_f g h_f^{-1} = I + c$ for a nonzero $c\in {\rm Per}^1(A)$. By hypothesis for some $k\geq 2$ there exists $d\in {\rm Per}^k(A) - {\rm Per}^1(A)$. Since $c\ne 0$, then $d+c\ne c$. We have $d+c\in {\rm Per}^k(A)-{\rm Per}^1(A)$ because $A^k(d+c) = A^k d + A^k c = d +c$ and if $d+c\in {\rm Per}^1(A)$ then $d+c = A(d+c) = Ad+c$ implying that $d\in{\rm Per}^1(A)$. Set $\theta = h_f^{-1}(d)$. Since
\[ h_f g h_f^{-1}(d) = (I+c)(d) = d+c,\]
we have $h_f g(\theta) = d+c$, and so $g(\theta) = h_f^{-1}(d+c)$.  Since $h_f$ is a homeomorphism, we have $g(\theta)\ne \theta$. Since $d,d+c\in {\rm Per}^k(A)-{\rm Per}^1(A)$, we have that $\theta,g(\theta)\in{\rm Per}^k(f) - {\rm Per}^1(f)$.

Suppose the orbits of $d$ and $d+c$ under $A$ are not distinct. Then there is $s\in{\mathbb N}$ such that $A^s d = d+c$, or $(A^s-I) d = c$. Since $A$ commutes with $A^s-I$ and $c\in {\rm Per}^1(A)$, we obtain
\[ (A^s-I)Ad = A(A^s-I)d = Ac = c.\]
Subtracting this from $(A^s-I)d=c$ gives $(A^s-I)Ad - (A^s-I)d = 0$. Factoring this gives
\[ (A^s-I)(A-I) d = 0.\]
Since $A^s-I$ and $A-I$ commute, we have either $(A-I)d = 0$ or $(A^s-I)d = 0$. The former is impossible because $d\not\in {\rm Per}^1(A)$. So $d\in {\rm Per}^s(A)$. Then $0=A^s d - d = c = 0$, a contradiction.

With the orbits of $d$ and $d+c$ under $A$ being distinct, the orbits of $\theta$ and $g(\theta)$ under $f$ are distinct because $h_f $ is a homeomorphism. Since $g\in Z(f)$, we have $f^k g = g f^k$. This implies that
\[ D_{g(\theta)}f^k D_\theta g = D_\theta g D_{f^k(\theta)} D_\theta f^k = D_\theta g D_\theta f^k.\]
This says that $D_{g(\theta)}f^k$ and $D_\theta f^k$ are similar. But this is impossible since $\theta$ and $g(\theta)$ belong to distinct $k$-periodic orbits of $f\in {\mathcal U}_k(A)$.
\end{proof}

Having shown that the homeomorphisms $h_f^{-1}(A^mJ^l+c)h_f$, $l\ne 0$, $c\in {\rm Per}^1(A)$, and $h_f^{-1}(I+c)h_f$ for nonzero $c\in {\rm Per}^1(A)$ are not diffeomorphisms, we can now in a position to state the sufficient conditions by which every $f\in {\mathcal U}_p(A)$ has trivial centralizer. The necessity of the order of $J$ being a power of two in Theorem \ref{precursoroddprime} is similar to that for Theorem \ref{precursorfixed} as described at the end of Section 4.

\begin{thm}\label{precursoroddprime} Let $A\in{\rm GL}(n,{\mathbb Z})$ be hyperbolic. If $C(A) = \langle A\rangle \times \langle J\rangle$, where the order of $J$ is $2^b$ for some $b\in{\mathbb N}$, and if
\[ \vert {\rm Per}^p(A) - {\rm Per}^1(A)\vert > \vert {\rm Per}^1(-I)\vert\]
for some odd prime $p$, then $Z(f)$ is trivial for all $f\in {\mathcal U}_p(A)$.
\end{thm}

\begin{proof} Let $p$ be an odd prime for which
\[ \vert {\rm Per}^p(A) - {\rm Per}^1(A)\vert > \vert {\rm Per}^1(-I)\vert\]
holds. For an arbitrary $f\in{\mathcal U}_p(A)$, let $g\in Z(f)$ be arbitrary. Since $C(A)=\langle A\rangle \times \langle J\rangle$, we have by Theorem \ref{nbhd1}  that
\[ h_f g h_f^{-1} = A^m J^l + c\]
where $m\in{\mathbb Z}$, $0\leq l\leq 2^b-1$, and $c\in {\rm Per}^1(A)$. By Lemma \ref{elimination}, it follows that $l=0$ since the order of $J$ is $2^b$ for some $b\in{\mathbb N}$ and $-I\in\langle J\rangle$. Thus
\[ h_f g h_f^{-1} = A^m + c.\]
Since $f^{-m}\in Z(f)$ and $h_f f^{-m}h_f^{-1} = A^{-m}$, we have that $f^{-m}g\in Z(f)$ and 
\[ h_f f^{-m} g h_f^{-1} = (A^{-m})(A^m + c) = I + c.\]
Since $\vert {\rm Per}^1(-I)\vert > 0$, it follows that ${\rm Per}^p(A)-{\rm Per}^1(A)\ne\emptyset$. Then Lemma \ref{identity2} implies that $c=0$. Hence
\[ h_f f^{-m} g h_f^{-1} = I.\]
This implies that $g = f^m$. With $g\in Z(f)$ being arbitrary, we have that $Z(f)$ is trivial. Since $f\in {\mathcal U}_p(A)$ is arbitrary, we have that $Z(f)$ is trivial for all $f\in {\mathcal U}_p(A)$.
\end{proof}

\section{Proof of Main Theorem}

The proof of Theorem \ref{t.open} will follow immediately by exhibiting for each $n=2,3,4$ the existence of an irreducible hyperbolic toral automorphism to which we can apply either the fixed point approach or the odd prime periodic point approach. We start with an irreducible polynomial
\[ p(x) = x^n+a_{n-1}x^{n-1}+\cdot\cdot\cdot+a_1x + a_0, \ a_i\in{\mathbb Z},\]
whose constant term $a_0$ is $\pm 1$, and whose signature $(r_1,r_2)$ satisfies $r_1+r_2-1=1$. The companion matrix for $p(x)$ is the ${\rm GL}(n,{\mathbb Z})$ matrix
\[ A = \begin{bmatrix} 0 & 1 & 0 & \dots & 0 \\ 0 & 0 & 1 & \dots & 0 \\ \vdots & \vdots & \vdots & \ddots & \vdots \\  0 & 0 & 0 & \dots & 1 \\-a_0 & -a_1 & -a_2 & \dots & -a_{n-1}\end{bmatrix}\]
and it satisfies $p_A(x) = p(x)$. We set ${\mathbb F}={\mathbb Q}(\lambda)$ for a root $\lambda$ of $p_A(x)$. Making use of known tables (see \cite{Coh,PZ}) that list, by means of the discriminant of ${\mathbb F}$, a basis of ${\mathfrak o}_{\mathbb F}$ in terms of $\lambda$, and a fundamental unit $\epsilon_{\mathbb F}$ in terms of the basis of ${\mathfrak o}_{\mathbb F}$, we check to see if
\[ {\mathbb Z}[\lambda] = {\mathfrak o}_{\mathbb F}{\rm \ and\ }\lambda=\epsilon_{\mathbb F},\]
as these do not always happen (see Examples \ref{not}, \ref{noninteger}, and \ref{another} in this paper). When both do, then ${\mathbb Z}[\lambda]^\times = {\mathfrak o}_{\mathbb F}^\times$, and we use Corollary \ref{cases234} to obtain
\[ C(A) = \langle A\rangle \times \langle J\rangle\]
where $\gamma(A) = \epsilon_{\mathbb F}$ and $\gamma(J) = \zeta_{\mathbb F}$. When $n=2,3$, we always have that $J=-I$, so that the order of $J$ is $2$. When $n=4$, we use an algorithm (described in \cite{PZ}) to determine the order of $\zeta_{\mathbb F}$ and hence that of $J$. We check this order to ensure that it is $2$, $4$, or $8$, and not the possibilities of $6$, $10$, or $12$. When the order of $J$ is $2^b$ for some $b\in{\mathbb N}$, we use the Smith normal form to compute the values of $\vert {\rm Per}^1(A)\vert$ and $\vert {\rm Per}^p(A)\vert$ for various odd primes $p$. We check to see if
\[ \vert {\rm Per}^1(A)\vert > \vert {\rm Per}^1(-I)\vert{\rm \ or \ }\vert {\rm Per}^p(A)-{\rm Per}^1(A)\vert > \vert {\rm Per}^1(-I)\vert.\]
We then apply Theorem \ref{precursorfixed} in the former, or Theorem \ref{precursoroddprime} in the latter for the smallest odd prime $p$ possible, to get an open subset ${\mathcal U}_1(A)$ or ${\mathcal U}_p(A)$ of ${\rm Diff}^1({\mathbb T}^n)$ with trivial centralizer.

\begin{ex} {\rm For $n=2$, we show that we can apply the fixed point approach to the companion matrix
\[ A  = \begin{bmatrix} 0 & 1 \\ 1 & 5\end{bmatrix}.\]
of the irreducible $p(x) = x^2-5x-1$ with signature $(r_1,r_2)$ satisfying $r_1+r_2-1=1$. None of the roots of $p_A(x)=p(x)$ have modulus one, and so $A$ is hyperbolic. A root of $p_A(x)$ is $\lambda = (5+\sqrt{29})/2$. The real quadratic field ${\mathbb F}={\mathbb Q}(\lambda)$ has discriminant $29$. A basis for ${\mathfrak o}_{\mathbb F}$ is $1$ and $\omega = (1+\sqrt{29})/2$, so that ${\mathfrak o}_{\mathbb F} = {\mathbb Z}[\omega]$. One verifies that ${\mathbb Z}[\lambda] = {\mathfrak o}_{\mathbb F}$ because of the change of basis
\[ \begin{bmatrix} 1 & 0 \\ 2 & 1 \end{bmatrix} \begin{bmatrix} 1 \\ \omega\end{bmatrix} = \begin{bmatrix} 1 \\ \lambda\end{bmatrix}\]
given by the ${\rm GL}(2,{\mathbb Z})$ matrix. A fundamental unit is $\epsilon_{\mathbb F} = (5+\sqrt{29})/2$, which is $\lambda$. Thus $C(A) = \langle A\rangle \times \langle -I\rangle$. By the Smith normal form, we have $\vert{\rm Per}^1(A)\vert = 5$ which is bigger than $\vert{\rm Per}^1(-I)\vert = 4$. Thus every $f$ in the open ${\mathcal U}_1(A)\subset {\rm Diff}^1({\mathbb T}^2)$ has trivial centralizer.
}\end{ex}

\begin{ex}{\rm For $n=3$, we show that we can apply the odd prime periodic point approach to the companion matrix
\[ A  = \begin{bmatrix}  0 & 1 & 0 \\ 0 & 0 & 1 \\ 1  & 0  & 1 \end{bmatrix}\]
of the irreducible polynomial $p(x) = x^3-x^2-1$ with signature $(r_1,r_2)$ satisfying $r_1+r_2-1=1$. None of the roots of $p_A(x) = p(x)$ have modulus one, so that $A$ is hyperbolic. For a root $\lambda$ of $p_A(x)$, the complex cubic field ${\mathbb F}={\mathbb Q}(\lambda)$ has discriminant -31. A basis for ${\mathfrak o}_{\mathbb F}$ is $1$, $\lambda$, and $\lambda^2$, so that ${\mathfrak o}_{\mathbb F} = {\mathbb Z}[\lambda]$. A fundamental unit is $\epsilon_{\mathbb F}=\lambda$. Thus $C(A) = \langle A\rangle\times\langle -I\rangle$. By the Smith normal form we have $\vert {\rm Per}^1(A)\vert=1$, $\vert {\rm Per}^3(A)\vert = 1$, $\vert{\rm Per}^5(A)\vert = 11$, so that $\vert {\rm Per}^5(A)-{\rm Per}^1(A)\vert=10$ which is bigger than  $\vert{\rm Per}^1(-I)\vert=8$. Thus every $f$ in the open ${\mathcal U}_5(A)\subset{\rm Diff}^1({\mathbb T}^3)$ has trivial centralizer.
}\end{ex}

\begin{ex}{\rm For $n=4$ we show that we can apply the odd prime periodic point approach to the companion matrix
\[ A = \begin{bmatrix} 0 & 1 & 0 & 0 \\ 0 & 0 & 1 & 0 \\ 0 & 0 & 0 & 1 \\ -1 & 2 & 0 & -2\end{bmatrix}\]
of the irreducible polynomial $p(x) = x^4+2x^3 - 2x+1$ with signature $(r_1,r_2)$ satisfying $r_1+r_2-1=1$. None of the roots of $p_A(x) = p(x)$ have modulus one, so that $A$ is hyperbolic. For $\lambda$ a root of $p_A(x)$, the totally complex quartic field ${\mathbb F} = {\mathbb Q}(\lambda)$ has discriminant 320. A basis for ${\mathfrak o}_{\mathbb F}$ is $1$, $\lambda$, $\lambda^2$, and $\lambda^3$, so that ${\mathfrak o}_{\mathbb F} = {\mathbb Z}[\lambda]$. A fundamental unit is $\epsilon_{\mathbb F} = \lambda$. Thus $C(A) = \langle A\rangle \times \langle J \rangle$ where $\gamma(J)=\zeta_{\mathbb F}$. Computation of the order of $\zeta_{\mathbb F}$ gives it as $4$, so that the order of $J$ has the form $2^b$ for some $b\in{\mathbb N}$. Specifically, we have
\[ J = \begin{bmatrix} 1 & -1 & -2 & -1 \\ 1 & -1 & -1 & 0 \\ 0 & 1 & -1 & 1 \\ 1 & -2  & 1 & 1\end{bmatrix}\]
which satisfies $J^2=-I$. By the Smith normal form, we have that $\vert {\rm Per}^1(A)\vert = 2$ and $\vert{\rm Per}^3(A)\vert = 26$ so that $\vert {\rm Per}^3(A) - {\rm Per}^1(A)\vert = 24$ is bigger than $\vert {\rm Per}^1(-I)\vert = 16$. Thus every $f$ in the open ${\mathcal U}_3(A)\subset{\rm Diff}^1({\mathbb T}^4)$ has trivial centralizer.
}\end{ex}

\section{Versatility and Limitations of Two Approaches}

In the proof of the Main Theorem in Section 6, we applied the fixed point approach when $n=2$ and the odd prime periodic point approach when $n=3$. We show here that we can apply the odd prime periodic point approach when $n=2$ and also the fixed point approach when $n=3$. We also show in $n=2$ through two examples that we can apply the both approaches when ${\mathbb Z}[\lambda]\ne {\mathfrak o}_{\mathbb F}$. We further exhibit in each of $n=2$ and $n=4$ an irreducible hyperbolic toral automorphism to which neither approach applies because either $C(A)\ne \langle A\rangle \times \langle J\rangle$ or because the order of $J$ is not the required power of $2$.

\begin{ex}{\rm For $n=2$, we show that we can apply the odd prime periodic point approach to the companion matrix
\[ A = \begin{bmatrix} 0 & 1 \\ 1 & 1\end{bmatrix}\]
of the irreducible polynomial $p(x) = x^2-x-1$ with signature $(r_1,r_2)$ satisfying $r_1+r_2-1=1$. None of the roots of $p_A(x)=p(x)$ have modulus one, so that $A$ is hyperbolic. A root of $p_A(x)$ is $\lambda = (1+\sqrt 5)/2$, and the real quadratic field ${\mathbb F}={\mathbb Q}(\lambda)$ has discriminant $5$. A basis for ${\mathfrak o}_{\mathbb F}$ is $1$ and $\lambda$, so that ${\mathbb Z}[\lambda]={\mathfrak o}_{\mathbb F}$. A fundamental unit is $\epsilon_{\mathbb F} = \lambda$. Thus $C(A) = \langle A\rangle \times \langle -I\rangle$. By the Smith normal form, we have $\vert {\rm Per}^1(A)\vert =1$, $\vert {\rm Per}^3(A)\vert = 4$, and $\vert {\rm Per}^5(A)\vert = 11$. Thus $\vert {\rm Per}^5(A) - {\rm Per}^1(A)\vert = 10$ which is bigger than $\vert {\rm Per}^1(-I)\vert=4$. Thus every $f$ in the open ${\mathcal U}_5(A)\subset{\rm Diff}^1({\mathbb T}^2)$ has trivial centralizer.
}\end{ex}

\begin{ex}{\rm For $n=3$, we show that we can apply the fixed point approach to the companion matrix 
\[ A = \begin{bmatrix} 0 & 1 & 0 \\ 0 & 0 & 1 \\ -1 & -6 & -4\end{bmatrix}\]
of the irreducible polynomial $p(x) = x^3+4x^2 + 6x+1$ whose signature $(r_1,r_2)$ satisfies $r_1+r_2-1=1$. None of the roots of $p_A(x)=p(x)$ have modulus one, so that $A$ is hyperbolic. For $\lambda$ a root of $p_A(x)$, the complex cubic field ${\mathbb F}={\mathbb Q}(\lambda)$ has discriminant $-139$. A basis for ${\mathfrak o}_{\mathbb F}$ is $1$, $\lambda$, and $1+2\lambda+\lambda^2$. One verifies that ${\mathbb Z}[\lambda] = {\mathfrak o}_{\mathbb F}$ because of the change of basis
\[ \begin{bmatrix} 1 & 0 & 0 \\ 0 & 1 & 0 \\ 1 & 2 & 1\end{bmatrix}  \begin{bmatrix} 1 \\ \lambda \\ \lambda^2\end{bmatrix} = \begin{bmatrix} 1 \\ \lambda \\ 1 + 2\lambda + \lambda^2\end{bmatrix}\]
given by a ${\rm GL}(3,{\mathbb Z})$ matrix. A fundamental unit is $\epsilon_{\mathbb F} = \lambda$. Thus $C(A) = \langle A\rangle \times \langle -I \rangle$. By the Smith normal form we have $\vert {\rm Per}^1(A)\vert = 12$ which is bigger than $\vert {\rm Per}^1(-I)\vert = 8$. Thus every $f$ in the open ${\mathcal U}_1(A)$ has trivial centralizer.
}\end{ex}

\begin{ex}\label{not}{\rm For $n=2$, we show that we can apply both approaches when  ${\mathbb Z}[\lambda]\ne {\mathfrak o}_{\mathbb F}$ but ${\mathbb Z}[\lambda]^\times = {\mathfrak o}_{\mathbb F}^\times$. Consider companion matrix
\[ A = \begin{bmatrix} 0 & 1 \\ 1 & 8 \end{bmatrix}\]
of the irreducible polynomial $p(x) = x^2 - 8x -1$ whose signature $(r_1,r_2)$ satisfies $r_1+r_2-1=1$. None of the roots of $p_A(x)=p(x)$ have modulus one, so that $A$ is hyperbolic. For the root $\lambda = 4+\sqrt{17}$ of $p_A(x)$, the real quadratic field ${\mathbb F}={\mathbb Q}(\lambda)$ has discriminant $17$. A basis for ${\mathfrak o}_{\mathbb F}$ is $1$ and $\omega = (1+\sqrt{17})/2$. It follows that ${\mathbb Z}[\lambda]\ne {\mathfrak o}_{\mathbb F}$ because $\omega\in{\mathfrak o}_{\mathbb F}$ while $\omega\not\in{\mathbb Z}[\lambda]$. A fundamental unit is $\epsilon_{\mathbb F} = 3+2\omega = \lambda$. Because
\[ \epsilon_{\mathbb F}^{-1} = -5+2\omega = -8 + (4+\sqrt{17}) \in {\mathbb Z}[\lambda],\]
it follows that ${\mathbb Z}[\lambda]^\times={\mathfrak o}_{\mathbb F}^\times$. Thus by Corollary \ref{cases234} we have that $C(A) = \langle A \rangle \times \langle -I\rangle$. By the Smith normal form we have that $\vert {\rm Per}^1(A)\vert = 8$ and $\vert {\rm Per}^3(A) - {\rm Per}^1(A)\vert =  528$ each of which is bigger than $\vert {\rm Per}^1(-I)\vert = 4$. Thus each $f$ in the open ${\mathcal U}_1(A)\subset{\rm Diff}^1({\mathbb T}^2)$ and each $f$ in the open ${\mathcal U}_3(A)\subset{\rm Diff}^1({\mathbb T}^2)$ has trivial centralizer.
}\end{ex}

\begin{ex}\label{noninteger}{\rm We show that we can apply both approaches when ${\mathbb Z}[\lambda]\ne{\mathfrak o}_{\mathbb F}$ and ${\mathbb Z}[\lambda]^\times\ne {\mathfrak o}_{\mathbb F}^\times$, but $\gamma(C(A)) = {\mathbb Z}[\lambda]^\times$. Consider the ${\rm GL}(2,{\mathbb Z})$ matrix
\[ A = \begin{bmatrix} 18 & 5 \\ 65 & 18\end{bmatrix}\]
whose characteristic polynomial is the irreducible $p_A(x) = x^2 - 36x -1$. The signature $(r_1,r_2)$ of $p_A(x)$ satisfies $r_1+r_2-1=1$. None of  the roots of $p_A(x)$ have modulus one, so that $A$ is hyperbolic. A root of $p_A(x)$ is $\lambda = 18+5\sqrt{13}$. The real quadratic field ${\mathbb F}={\mathbb Q}(\lambda)$ has discriminant $13$. A basis for ${\mathfrak o}_{\mathbb F}$ is $1$ and $\omega = (1+\sqrt 13)/2$. Then ${\mathbb Z}[\lambda] \ne {\mathfrak o}_{\mathbb F}$ because $\omega\not\in {\mathbb Z}[\lambda]$ while $\omega\in{\mathfrak o}_{\mathbb F}$. A fundamental unit is $\epsilon_{\mathbb F} = 1+\omega = (3+\sqrt{13})/2$.

Recognizing that $\epsilon_{\mathbb F}^3 = \lambda$ suggests that either $A$ has a cube root in ${\rm GL}(2,{\mathbb Z})$ or that  $\gamma(C(A)) = {\mathbb Z}[\lambda]^\times$. Now because both $\epsilon_{\mathbb F}$ and $\epsilon_{\mathbb F}^2 = (11+3\sqrt{13})/2$ are not in ${\mathbb Z}[\lambda]$, but $\lambda^{-1} = -18+5\sqrt{13}$ is in ${\mathbb Z}[\lambda]$, we have that
\[ {\mathbb Z}[\lambda]^\times = \{ \pm \epsilon_{\mathbb F}^{3m}: m\in{\mathbb Z}\}.\]
By the outline of the proof of Theorem \ref{C(A)structure} we know that ${\mathbb Z}[\lambda]^\times \subset \gamma(C(A))\subset {\mathfrak o}_{\mathbb F}^\times$. To show that $\gamma(C(A)) = {\mathbb Z}[\lambda]^\times$, we will show that for any $B\in C(A)$ we have $\gamma(B) \in {\mathbb Z}[\lambda]^\times$. For $B\in C(A)$ there is $v(x)\in {\mathbb Q}[x]$ such that $B=v(A)$. Since $\gamma(C(A))\subset{\mathfrak o}_{\mathbb F}^\times$, it follows that $\gamma(v(A)) = v(\lambda) = \pm \epsilon_{\mathbb F}^s$ for some $s\in {\mathbb Z}$. Write $s=3m+j$ for $m\in{\mathbb Z}$ and $j=0,1,2$. Suppose that $j\ne 0$. If $m<0$ let $w(x) = x$, if $m=0$ let $w(x)=1$, or if $m>0$ let $w(x) = x-36$. Then $w(A) = A$ when $m<0$, $w(A) = I$ when $m=0$, and $w(A)=A^{-1}$ when $m>0$, where the latter follows because $A^2-36A-I=0$ implies $A^{-1} = A-36I$. Thus $w(\lambda) = \lambda$ when $m<0$, $w(\lambda) = 1$ when $m=0$, and $w(\lambda) = \lambda^{-1}$ when $m>0$. If $\gamma(v(B)) = -\epsilon_{\mathbb F}^s$ let $y(x) = -1$, and if $\gamma(v(B)) = \epsilon_{\mathbb F}^s$ let $y(x) = 1$. It follows that
\begin{align*} \gamma\big( y(A)[w(A)]^{\vert m\vert} v(A)\big) & = y(\lambda) [w(\lambda)]^{\vert m\vert} v(\lambda)  \\
& = (\pm 1)(\lambda^{-m})(\pm \epsilon_{\mathbb F}^{3m+j}) \\
& = \epsilon_{\mathbb F}^j.
\end{align*}
By the Division Algorithm, there are $q(x),r(x)\in {\mathbb Q}[x]$ such that
\[ y(x)[w(x)]^{\vert m \vert} v(x) = q(x)p_A(x) + r(x)\]
where the degree of $r(x)$ is less than $2$. Because $p_A(A) = 0$ we have $y(A)[w(A)]^{\vert m\vert} v(A) = r(A)$ and $y(\lambda)[w(\lambda)]^{\vert m\vert}v(\lambda) = r(\lambda)$. If $r(\lambda) = \epsilon_{\mathbb F}$ then $r(x) = -3/10 + x/10$, and if $r(\lambda) = \epsilon_{\mathbb F}^2$, then $r(x) = 1/10 + 3x/10$. For the first we have 
\[ r(A) = \frac{-3}{10} I + \frac{1}{10} A = \begin{bmatrix} 3/2 & 1/2 \\ 13/2 & 3/2\end{bmatrix},\]
while for the second we have 
\[ r(A) = \frac{1}{10} I + \frac{3}{10} A = \begin{bmatrix} 11/2 & 3/2 \\ 39/2 & 11/2\end{bmatrix},\]
neither of which is in ${\rm GL}(2,{\mathbb Z})$. (Note that the first $r(A)$ satsifies $[r(A)]^3=A$, and that the second $r(A)$ is the square of the first.) Since all of $y(A)$, $[w(A)]^{\vert m\vert}$, and $B=v(A)$ are in ${\rm GL}(2,{\mathbb Z})$ it follow that $r(A)$ is also in ${\rm GL}(2,{\mathbb Z})$. This contradiction shows that $j=0$, so that indeed $\gamma(C(A)) = {\mathbb Z}[\lambda]^\times$.

This completely identifies $\gamma(C(A))$. Since $\gamma(A) = \lambda$ and $\gamma(-I) = -1$, we have that $C(A) = \langle A\rangle\times\langle -I\rangle$. By the Smith normal form we have that $\vert {\rm Per}^1(A)\vert = 36$ and $\vert {\rm Per}^3(A) - {\rm Per}^1(A)\vert = 46728$ both of which are bigger than $\vert {\rm Per}^1(-I)\vert =4$. Thus each $f$ in the open ${\mathcal U}_1(A)$ and each $f$ in the open ${\mathcal U}_3(A)$ has trivial centralizer.
}\end{ex}

\begin{ex}\label{another}{\rm For $n=2$ we show there exists an irreducible hyperbolic $A$ to which we cannot apply the fixed point approach nor the odd prime periodic point approach because $C(A) \ne \langle A \rangle \times \langle J\rangle$. Consider the ${\rm GL}(2,{\mathbb Z})$ matrix
\[ A = \begin{bmatrix} 2 & 5 \\ 5 & 12\end{bmatrix}\]
whose characteristic polynomial is the irreducible $p_A(x) = x^2-14x -1$. The signature $(r_1,r_2)$ of $p_A(x)$ satisfies $r_1+r_1-1=1$. None of the roots of $p_A(x)$ have modulus one, so that $A$ is hyperbolic. A root of $p_A(x)$ is $\lambda = 7+5\sqrt 2$. The real quadratic field ${\mathbb F}={\mathbb Q}(\lambda)$ has discriminant $8$. A basis for ${\mathfrak o}_{\mathbb F}$ is $1$ and $\omega = \sqrt 2$. Then ${\mathbb Z}[\lambda]\ne {\mathfrak o}_{\mathbb F}$ because $\sqrt 2\not\in {\mathbb Z}[\lambda]$ while $\sqrt 2 \in {\mathfrak o}_{\mathbb F}$. A fundamental unit is $\epsilon_{\mathbb F} = 1 + \sqrt 2$. Then ${\mathbb Z}[\lambda]^\times \ne {\mathfrak o}_{\mathbb F}^\times$ because $1+\sqrt 2\not\in {\mathbb Z}[\lambda]$.

Recognizing that $\epsilon_{\mathbb F}^3 = \lambda$ suggests that either $A$ has a cube root in ${\rm GL}(2,{\mathbb Z})$ or that $\gamma(C(A)) = {\mathbb Z}[\lambda]^\times$. Under the isomorphism $\gamma$ we have for $v(x) = (1/5)x - 2/5$ that $v(\lambda) = \epsilon_{\mathbb F}$, so that
\[ v(A) =\frac{1}{5} A - \frac{2}{5}I = \begin{bmatrix} 0 & 1 \\ 1 & 2\end{bmatrix} \in {\rm GL}(2,{\mathbb Z}),\]
which indeed satisfies $[v(A)]^3 = A$. Now $v(A)$ commutes with $A$, so that $v(A)\in C(A)$. Because $\gamma(v(A)) = \epsilon_{\mathbb F}$ and $\gamma(-I) = -1$, we have that $\gamma(C(A)) = {\mathfrak o}_{\mathbb F}^\times$, and so
\[ C(A) = \langle v(A) \rangle \times \langle -I\rangle \ne \langle A\rangle \times \langle -I\rangle.\]
For any $f\in {\mathcal U}(A)$ there is the homeomorphism $g= h_f^{-1}v(A)h_f\in H_f(A)$. Because $C(A)\ne \langle A\rangle \times \langle -I\rangle$, we cannot eliminate $g$ as a diffeomorphism by the fixed point approach nor by the odd prime periodic point approach, even though by the Smith normal form we have that both $\vert {\rm Per}^1(A)\vert = 14$ and $\vert {\rm Per}^3(A) - {\rm Per}^1(A)\vert = 2772$ are bigger than $\vert {\rm Per}^1(-I)\vert=4$. Although we know that the elements of a dense subset of ${\mathcal U}(A)\setminus\{A\}$ have trivial centralizer (by \cite{BCW1}), we do not know if there is an open subset of ${\mathcal U}(A)\setminus\{A\}$ whose elements have trivial centralizer.
}\end{ex}

\begin{ex}{\rm For $n=4$ we show that there is an irreducible hyperbolic $A$ to which we cannot apply the fixed point approach nor the odd prime periodic point approach because the order of $J$ is not a power of $2$. Consider the companion matrix
\[ A = \begin{bmatrix} 0 & 1 & 0 & 0 \\ 0 & 0 & 1 & 0 \\ 0 & 0 & 0 & 1 \\ -1 & -3 & -10 & -6\end{bmatrix}\]
of the irreducible $p(x) = x^4 + 6x^3 + 10x^2 + 3x +1$ whose signature $(r_1,r_2)$ satisfies $r_1+r_2-1=1$. None of the roots of $p_A(x) = p(x)$ have modulus one, so that $A$ is hyperbolic. For $\lambda$ a root of $p_A(x)$, the totally complex quartic field ${\mathbb F}={\mathbb Q}(\lambda)$ has discriminant $549$. A basis for ${\mathfrak o}_{\mathbb F}$ is $1$, $\lambda$, $-1+2\lambda+\lambda^2$, and $3\lambda+4\lambda^2+\lambda^3$. One verifies that ${\mathbb Z}[\lambda] = {\mathfrak o}_{\mathbb F}$ because of the change of basis
\[ \begin{bmatrix} 1 & 0 & 0 & 0 \\ 0 & 1 & 0 & 0 \\ -1 & 2 & 1 & 0 \\ 0 & 3 & 4 & 1\end{bmatrix} \begin{bmatrix} 1 \\ \lambda \\ \lambda^2 \\ \lambda^3\end{bmatrix} = \begin{bmatrix} 1 \\ \lambda \\ -1 + 2\lambda + \lambda^2 \\ 3\lambda + 4\lambda^2 + \lambda^3\end{bmatrix}\]
given by a ${\rm GL}(4,{\mathbb Z})$ matrix. A fundamental unit is $\epsilon_{\mathbb F} = \lambda$. Thus $C(A) = \langle A\rangle \times \langle J\rangle$ where $\gamma(A) = \epsilon_{\mathbb F}$ and $\gamma(J) = \zeta_{\mathbb F}$. Computation of the order of $\zeta_{\mathbb F}$ gives it as $6$. Specifically, we have
\[ J = \begin{bmatrix} 0 & -3 & -1 & 0 \\ 0 & 0 & -3 & -1 \\ 1 & 3 & 10 & 3 \\ -3 & -8 & -27 & -8\end{bmatrix}\]
which satisfies $J^3=-I$. By the Smith normal form we have that both $\vert {\rm Per}^1(A)\vert = 21$ and $\vert {\rm Per}^3(A) - {\rm Per}^1(A)\vert = 546$ are greater than $\vert {\rm Per}^1(-I)\vert=16$. But we cannot apply the fixed point approach nor the odd prime periodic point approach  because the order of $J$ is not a power of $2$. Although we do know that the elements of a dense subset of ${\mathcal U}(A)\setminus\{A\}$ have trivial centralizer (by \cite{BCW1}), we do not know if there is an open subset of ${\mathcal U}(A)\setminus\{A\}$ whose elements have trivial centralizer.
}\end{ex}


\end{document}